\documentclass[12pt]{amsart}
\usepackage{amssymb,latexsym}
\usepackage{amsfonts}
\usepackage{amsmath}
\usepackage[colorlinks,linkcolor=blue,anchorcolor=blue,citecolor=blue]{hyperref}
\usepackage{algorithm}
\usepackage{enumerate}
\usepackage{algpseudocode}
\usepackage{verbatim}
\usepackage{graphicx}

\newcommand{\tr}{{\rm Tr}}

\newcommand{\gZ}{{\mathrm Z}}

\newcommand{\mM}{{\mathbf M}}

\newcommand{\C}{{\mathbb C}}

\newcommand{\F}{{\mathbb F}}
\newcommand{\gN}{\mathrm{N}}

\newtheorem{theorem}{Theorem}[section]
\newtheorem{lemma}[theorem]{Lemma}

\theoremstyle{definition}
\newtheorem{definition}[theorem]{Definition}
\newtheorem{example}[theorem]{Example}

\newtheorem{corollary}[theorem]{Corollary}

\theoremstyle{remark}
\newtheorem{remark}[theorem]{Remark}

\numberwithin{equation}{section}

\begin{document}

\title{A new criterion on $k$-normal elements over finite fields}

\author{Aixian Zhang}
\address{Department of Mathematical Sciences, Xi'an  University of Technology,
Shanxi, 710054, China.}
\email{zhangaixian1008@126.com}

\author{Keqin Feng}
\address{Department of Mathematical
Sciences, Tsinghua University,
 Beijing, 100084, China. }
\email{kfeng@math.tsinghua.edu.cn}





\begin{abstract}
The notion of normal elements for finite fields extension has been generalized as $k$-normal
elements by Huczynska et al. [3]. The number of $k$-normal elements for a fixed finite
field extension has been calculated and estimated [3], and several methods to construct
$k$-normal elements have been presented [1,3].
Several criteria on $k$-normal element have been given [1,2]. In this paper we present a new criterion
on $k$-normal elements by using idempotents and show some examples. Such criterion has been given for usual normal element
before [6].
\end{abstract}

\maketitle



\section{Introduction}\label{sec-one}
Let $q=p^m$ where $p$ is a prime number, $m \geq 1, \ \F_q$ be the finite field
with $q$ elements, $\F^{\times}_q=\F_q \backslash \{0\}.$ For $ n \geq 1$ and
$Q=q^n, \alpha \in \F^{\times}_Q$ is called a normal element for extension
$\F_Q / \F_q$ if $\gN=\{ \alpha, \alpha^q, \alpha^{q^{2}}, \cdots, \alpha^{q^{n-1}}\}$
is a basis of $\F_Q  $ over $\F_q$ ( $\gN$ is called a normal basis for $\F_Q / \F_q$).
For a normal element $\alpha$ of $\F_Q / \F_q ,$ the minimal polynomial $f_{\alpha}(x) \in \F_q[x]$
of $\alpha$ over $\F_q$ is called a normal polynomial for $\F_Q / \F_q ,$ which is a monic
irreducible polynomial in $\F_q[x] $ with degree $n.$ Normal bases have many applications
including coding theory, cryptography and communication theory (\cite{LN,MP}). It is proved that
$\alpha \in \F^{\times}_Q$ is a normal element for $\F_Q / \F_q $ if and only if
\begin{equation}\label{equ-gcd}
\gcd(g_{\alpha}(x), x^n -1)=1, \quad  g_{\alpha}(x)=\sum^{n-1}_{i=0}\alpha^{q^i}x^{n-i-1}
\end{equation}
see (\cite{LN},Theorem 2.39).

The following definition given by Huczynska et al. \cite{HMPT} is a generalization of normal
elements.

\begin{definition}(see \cite{HMPT}.)
Let $q=p^m,Q=q^n$ and $0 \leq k \leq  n-1.$ An element $\alpha \in \F^{\times}_Q$
is called a $k$-normal element for $\F_Q / \F_q $ if the degree of $\gcd(g_{\alpha}(x), x^n -1)$
is $k$.
\end{definition}

With this terminology, a normal element is just $0$-normal. As shown in the normal
element case \cite{MP}, the $k$-normal elements can be used to reduce the multiplication
process in finite fields \cite{Negre}.

The number of $k$-normal elements for extension $\F_Q / \F_q $ has been calculated and
estimated in \cite{HMPT} and several methods to construct $k$-normal elements have been
presented in \cite{AM} and \cite{AT}. As the normal element case, the $k$-normal elements
can be characterized by using $q$-linearlized
polynomial theory [2,3]. Now we briefly introduce such characterization.

A $q$-linearlized polynomial ($q$-polynomial in brief) is a polynomial in the following form
$$L(x)=a_0 x +a_1 x^q + \cdots +a_m x^{q^{m}} \ \ \ (a_i \in \F_{q}).$$
Let $\mathcal{F}_q[x]$ be the set of all $q$-polynomials. Then $\mathcal{F}_q[x]$ is a ring with
respect to the ordinary addition and the following multiplication $\otimes$ :
$$
L(x) \otimes K(x)=L(K(x)) \ \ \ (\mbox{composition}).
$$
One of basic facts on $\mathcal{F}_q[x]$ is that the mapping
\begin{equation}
\varphi:\mathbb{F}_{q}[x] \longrightarrow  \mathcal{F}_q[x], \quad
\ \sum^m_{i=0}a_ix^i \mapsto \sum^m_{i=0}a_ix^{q^i} \ (a_i \in \mathbb{F}_{q})
\end{equation}
is an isomorphism of rings. Therefore $\mathcal{F}_q[x]$ is a principal ideal domain with identity $x.$
We use the notation $\parallel$
to express the divisibility in $\mathcal{F}_q[x]$. Namely, for $L(x)$
and $M(x)$ in $\mathcal{F}_q[x]$, $L(x) \parallel M(x)$ means that
$L(x) \neq 0$ and there exists $N(x) \in \mathcal{F}_q[x]$ such that $M(x)=L(x) \otimes N(x)=N(x) \otimes L(x).$

Let $n \geq 1, \alpha \in \mathbb{F}^{\times}_{Q}.$ The set
$$
I_\alpha=\{M(x) \in \mathcal{F}_q[x]: M(\alpha)=0\}
$$
is a nonzero ideal of $\mathcal{F}_q[x]$ since $x^{q^n}-x \in I_\alpha.$
The monic generator $M_\alpha(x)$ of the ideal $I_\alpha$
is called the minimal $q$-polynomial of $\alpha.$ Particularly, $M_\alpha(x)$ is an irreducible polynomial
in $\mathcal{F}_q[x]$ and $M_\alpha(x) \parallel x^{q^n}-x .$ Moreover for any
$L(x) \in \mathcal{F}_q[x], \ L(\alpha)=0$
if and only if $M_\alpha(x) \parallel L(x).$

\begin{lemma}([3, Theorem 3.2])\label{lem-equ}
Let $q=p^m, Q=q^n$ and $0 \leq k\leq n-1.$ The following statements for $\alpha \in \F^{\times}_{Q}$
are equivalent to each other

(I) $\alpha$ is a $k$-normal element for $\F_{Q}  / \F_{q};$

(II) The degree of the minimal $q$-polynomial $M_\alpha(x) \in \mathcal{F}_q[x]$ over $\F_{q}$ is $q^{n-k};$
\end{lemma}

(III) The dimension of the $\F_{q}$-vector subspace $V_{\alpha}$ of $\F_{Q}$ spanned by
$\{ \alpha, \alpha^{q},\ldots,\alpha^{q^{n-1}}\}$ is $n-k$ and $\{ \alpha, \alpha^{q},\ldots,\alpha^{q^{n-k-1}}\}$
is a $\F_{q}$-basis of $V_{\alpha}$.

Let $n=p^{t}n^{\prime}, p \nmid n^{\prime}.$ Then $x^n -1$ is decomposed in $\F_{q}[x]$ as
\begin{equation}\label{eqn-qdecom}
x^n -1=(x^{n^{\prime}} -1)^{p^{t}}=(p_1(x)p_2(x) \cdots p_s(x))^{p^{t}},
\end{equation}
where $p_i(x) \ (1 \leq i \leq s) $ are distinct monic irreducible polynomials in $\F_{q}[x]$.
By the isomorphism $\varphi$ in (2), $x^{q^n}-x$ has the following corresponding decomposition in $\mathcal{F}_q[x]:$
$$x^{q^n}-x=(P_1(x) \otimes P_2(x) \otimes \cdots  \otimes P_s(x))^{p^{t}},$$
where $P_i(x)=\varphi(p_i(x))\ (1 \leq i \leq s) $ are distinct monic irreducible $q$-polynomials in $\mathcal{F}_q[x]$
and for $L(x) \in \mathcal{F}_q[x]$ and $l \geq 1, L(x)^{l}$
means $L(x) \otimes L(x) \otimes \cdots  \otimes L(x)  $ \
($l$ copies).

For $\alpha \in \F^{\times}_{Q},$ the minimal $q$-polynomial $M_\alpha(x)$ is a divisor of $x^{q^n}-x$ in $\mathcal{F}_q[x].$
Therefore $M_\alpha(x)=\varphi(m_\alpha(x))$ for a divisor $ m_\alpha(x)$ of $x^n -1$ in $\F_{q}[x].$ From the definition of
$M_\alpha(x)$ and the isomorphism $\varphi$ between $\F_{q}[x]$ and $\mathcal{F}_q[x]$ we get the following result.

\begin{lemma}
Let $x^n -1$ be decomposed by formula (\ref{eqn-qdecom}) in $\F_q[x], m(x)$ is a monic divisor of $x^n -1$
in $\F_{q}[x].$ Let $M(x)=\varphi(m(x))$ and $M_i(x)=\varphi(\frac{m(x)}{p_i(x)})$ if $p_i(x) | m(x).$
Then $M(x)$ is the minimal $q$-polynomial of $\alpha$ if and only if $M(\alpha)=0$ and
for each $p_i(x)|m(x), M_i(\alpha) \neq 0.$
\end{lemma}

Particularly, if $\gcd(n,p)=1,$ then the decomposition  (\ref{eqn-qdecom}) becomes
\begin{equation}\label{eqn-decom}
x^n -1=p_1(x)p_2(x) \cdots p_s(x).
\end{equation}

For $\alpha \in \F^{\times}_{Q},$ the minimal $q$-polynomial $M_\alpha(x)$ has the form
$$
M_\alpha(x)=M_{\Delta}(x)=\otimes_{i \in \Delta}P_i(x)
$$
where $\Delta$ is a subset of $\{ 1,2,\cdots,s\}.$ In this case, $M_\alpha(x)$ can be described
by the following way.

\begin{lemma}\label{lem-cri}
Suppose that $Q=q^n, (n,q)=1$ and $x^n -1$ has decomposition formula (\ref{eqn-decom}) where $p_i(x) \ (1 \leq i \leq s)$
are distinct monic irreducible polynomials in $\F_q[x].$ Let
$$m_i(x)=\frac{x^n -1}{p_i(x)}, \quad M_i(x)=\varphi(m_i(x)) \ (1 \leq i \leq s).$$
For $\alpha \in \F^{\times}_{Q},$ let
$$\Delta=\Delta(\alpha)=\{ i: 1 \leq i \leq s, M_i(\alpha) \neq 0 \}.$$
Then the minimal $q$-polynomial $M_\alpha(x)$ of $\alpha$ is $M_{\Delta}(x)=\bigotimes_{i \in \Delta}P_i(x)$
and $\alpha$ is a $k$-normal element for $\F_{Q}/ \F_{q}$ where $k=n-\sum_{i \in \Delta }\deg p_i(x).$
\end{lemma}

\begin{proof}
For each $i,1 \leq i \leq s$,
\begin{eqnarray*}
P_i(x) \| M_{\alpha}(x) & \Longleftrightarrow & M_{\alpha}(x) \nparallel \frac{x^{q^n}-x}{P_i(x)}
=M_i(x) \in \mathcal{F}_q[x] (\ \mbox{since} \  x^{q^n}-x=\otimes^{s}_{i=1}P_i(x)) \\
& \Longleftrightarrow & M_i(\alpha) \neq 0  \Longleftrightarrow  i \in \Delta.
\end{eqnarray*}
Therefore $M_{\alpha}(x)=\prod\limits_{i \in \Delta}P_i(x).$ Since
$\deg M_{\alpha}(x)=\prod\limits_{i \in \Delta} \deg P_i(x)=q^{\sum_{i \in \Delta} \deg p_i(x)},$
by Lemma \ref{lem-equ} we know that $\alpha$ is a $k$-normal element for extension $\F_{Q}/ \F_{q}$
where $k=n-\sum\limits_{i \in \Delta }\deg p_i(x)=\sum^{s}\limits_{ \scriptstyle i=1 \atop \scriptstyle i \notin \Delta }\deg p_i(x).$
\end{proof}

Lemma \ref{lem-cri} presents a method to determine the normality $k$ and the minimal $q$-polynomial of an
element $\alpha \in \F^{\times}_{Q}$ provided we know the decomposition formula (4) in the $\gcd(n,q)=1$ case.
In this paper we present a new method to determine the normality and the minimal $q$-polynomial $M_{\alpha}(x)$
of $\alpha \in \F^{\times}_{Q},$ essentially by the partition of $\gZ_n=\mathbb{Z}/ n\mathbb{Z}$ into
$q$-classes without using the explicit form of the irreducible factors $p_i(x) \ (1 \leq i \leq s)$ of $x^n -1.$
 We explain this idempotent method in Section \ref{sec-two} and show several examples in Section \ref{sec-three}.

\section{Main result}\label{sec-two}

Let $q=p^m, Q=q^n$ and $\gcd(n,p)=1.$ A criterion on normal element for extension $\F_{Q}/ \F_{q}$ has
been given in [6] by using idempotents in semisimple $\F_{q}$-algebra $A=\F_{q}[x]/(x^n -1).$ In this section
we generalize this method to determine the normality $k$ and  $M_{\alpha}(x)$ of any $\alpha \in  \F^{\times}_{Q}.$

By assumption $\gcd(n,p)=1, x^n -1$ has the decomposition (4) in $\F_{q}[x]:$
$$x^n -1=p_i(x)p_2(x) \cdots p_s(x)$$
where $p_i(x) \ (1 \leq i \leq s)$ are distinct monic irreducible polynomials in $\F_{q}[x].$ Let
$$n_i = \deg p_i(x), l_i(x)=\frac{x^n -1}{p_i(x)}, L_i(x)=\varphi(l_i(x)) \ (1 \leq i \leq s).$$
Then $n_1 + n_2 +\cdots+n_s =n, \deg L_i(x)=q^{n_{i}} \  (1 \leq i \leq s).$ By the Chinese Remainder Theorem,
 $A=\F_{q}[x]/(x^n -1)$ is a direct sum of finite fields:
 $$A \cong \bigoplus^{s}_{i=1}\frac{\F_{q}[x]}{(p_i(x))} \cong  \bigoplus^{s}_{i=1} \F_{Q_{i}} \quad (Q_i=q^{^{n_{i}}}).$$
It is well known that zeros and degree of $P_i(x)$ can be described by $q$-classes of $\mathrm{Z}_n=\mathbb{Z}/ n\mathbb{Z}$.

\begin{definition}
Let $n \geq 2, q=p^m$ and $\gcd(n,p)=1.$ Two elements $a$ and $b$ in
$\mathrm{Z}_n=\mathbb{Z}/ n\mathbb{Z}=\{ 0,1,2,\cdots,n-1\}$ are called $q$-equivalent if there exists a positive
$i \in \mathbb{Z}$ such that $a \equiv bq^i \ (\bmod  n).$

This is an equivalent relation on $\mathrm{Z}_n$
and $\mathrm{Z}_n$ is partited into $q$-equivalent classes.
\begin{eqnarray*}
\mathcal{A}_1 &=& \{a_1=0\},   | \mathcal{A}_1 |=n_1 =1\\
\mathcal{A}_2 &=& \{a_2, a_2 q, \cdots, a_2 q^{n_{2}-1}\}, | \mathcal{A}_2 |=n_2  \\
&&\vdots \\
\mathcal{A}_s &=& \{a_s, a_s q, \cdots, a_s q^{n_s -1} \},| \mathcal{A}_s |=n_s
\end{eqnarray*}
where for $1 \leq i \leq s, n_i$ is the least positive integer such that $a_i q^{n_{i}} \equiv a_i \ (\bmod  n).$
\end{definition}

Let $\alpha$ be a primitive $n$-th root of 1 in the algebraic closure of $\F_{q}.$ Then for each $i,1 \leq i \leq s,$
$$\mathcal{S}_i = \{ \alpha^{\lambda}: \lambda \in \mathcal{A}_i\}=\{ \alpha^{a_i},\alpha^{a_i q},\ldots,\alpha^{a_i q^{n_i -1}}\}$$
are the set of zeros of a monic irreducible polynomial $p_i(x)$ in $\F_{q}[x]$ with degree $n_i.$ And $x^n -1$ is
decomposed in $\F_{q}[x]$ as (4).

Now we introduce the system of orthogonal (minimal) idempotents in ring $A=\F_{q}[x]/(x^n -1).$ Consider the natural
isomorphism of rings
\begin{eqnarray*}
\pi: A & \longrightarrow &\frac{\F_{q}[x]}{(p_1(x))} \oplus \frac{\F_{q}[x]}{(p_2(x))}
\oplus \cdots \oplus \frac{\F_{q}[x]}{(p_s(x))}, \frac{\F_{q}[x]}{(p_i(x))}=\F_{Q_i}, Q_i=q^{n_i} \\
f(x)&=& (f(x)(\bmod  p_1(x)),f(x)(\bmod  p_2(x)),\cdots,f(x)(\bmod  p_s(x))).
\end{eqnarray*}

\begin{definition}
Let $v_1=(1,0,\ldots,0),v_2=(0,1,\ldots,0),\cdots,v_s=(0,0,\ldots,1)$ be elements in
$$\frac{\F_{q}[x]}{(p_1(x))} \oplus \frac{\F_{q}[x]}{(p_2(x))}\oplus \cdots
\oplus \frac{\F_{q}[x]}{(p_s(x))}
=\F_{Q_1} \oplus \F_{Q_2} \oplus \cdots \oplus \F_{Q_s}.$$
Let
$$e_i(x)=\pi^{-1}(v_i) \ (1 \leq i \leq s).$$
Namely, $e_i(x) \ (1 \leq i \leq s)$ are determined by
\begin{equation}\label{eqn-cong}
e_i(x) \equiv \delta_{ij} \ (\bmod  p_j(x)) \ (1 \leq j \leq s),
\end{equation}
where $\delta_{ij}=1$ for $i=j$ and $\delta_{ij}=0$ otherwise. $\{ e_i(x) \ (1 \leq i \leq s)\}$
is called the system of orthogonal (minimal) idempotents of $A,$ since the following relationships hold
$$
e_i(x)e_j(x)=\delta_{ij} e_i(x), \  \sum^{s}_{i=1}e_i(x)=1 \quad \ (1 \leq i,j \leq s).
$$
\end{definition}

Now we present our main result which shows that the minimum $q$-polynomial and the normality of
$\alpha \in \F^{\times}_{Q}$ can be determined by using $e_i(x) \ (1 \leq i \leq s).$

\begin{theorem}\label{thm-main}
Let $q=p^m, Q=q^n$ and $\gcd (n,p)=1.$ Let $x^n -1$ be decomposed as
$x^n -1=p_i(x)p_2(x) \cdots p_s(x)$  in $\F_{q}[x]$ and the idempotents
$\{e_i(x) \ (1 \leq i \leq s)\}$ be defined by congruence equation (\ref{eqn-cong}).
Let $E_i(x)=\varphi(e_i(x))$ and $P_i(x)=\varphi(p_i(x)) \ (1 \leq i \leq n).$
For any $\alpha \in \F^{\times}_{Q},$ let
$$
\Delta=\Delta(\alpha)=\{ i: 1 \leq i \leq s,\  E_i(\alpha) \neq 0 \}.
$$
Then the minimal $q$-polynomial of $\alpha$ is $M_{\Delta}(x)=\bigotimes_{i \in \Delta}P_i(x)$
and $\alpha$ is a $k$-normal element for extension $\F_{Q}/ \F_{q}$ where $k$ (the normality of $\alpha$)
is given by $k=n-\sum_{i \in \Delta}\deg p_i(x).$
\end{theorem}

\begin{proof}
Let $m_i(x)=\frac{x^n -1}{p_i(x)}, M_i(x)=\varphi(m_i(x)) (1 \leq i \leq s).$ It has been proved in the proof
of [6, Theorem 2] that for each $\alpha \in \F^{\times}_{Q}$ and $1 \leq i \leq s, M_i(\alpha)\neq 0$
if and only if $E_i(\alpha) \neq 0.$ Then the conclusion can be derived directly from Lemma \ref{lem-cri}.
\end{proof}

The idempotents $e_1(x),e_2(x),\ldots,e_s(x)$ are determined by the congruence equations (\ref{eqn-cong}),
but the following method is easier to calculate $e_i(x) \ (1 \leq i \leq s)$ in certain cases.

\begin{theorem}(\cite{ZF},Theorem 3)\label{thm-ZF} Suppose that $Q=q^n, \gcd(n,q)=1.$ Let
$\mathcal{A}_i (1 \leq i \leq s)$ be the partition of $\mathrm{Z}_n$ into $q$-classes, $e_i(x)$
and $p_i(x) \ (1 \leq i \leq s)$ be the corresponding idempotents of $\F_q[x]/(x^n -1)$ and monic irreducible
factors of $x^n -1$ in $\F_q[x].$ Let $\zeta$ be a primitive $n$-th root of 1 in the algebraic closure of
$\F_q.$ For each $i (1 \leq i \leq s),$  we take $\alpha_i$ to be a zero of $p_i(x)$ which means
$\alpha_i=\zeta^{a_i}$ for some $a_i \in \mathcal{A}_i.$ Let
\begin{equation}
\varepsilon_i(x)=\sum_{a \in \mathcal{A}_i}x^a \quad (1 \leq i \leq s).
\end{equation}
And $\mathbf{M}$ is an $s \times s$ matrix over $\F_{q}$ defined by
$$\mathbf{M}=(\varepsilon_i(\alpha_j))_{1 \leq i,j \leq s}.$$
Then $\det(\mathbf{M})\neq 0$ and

\begin{eqnarray*}
\left(
          \begin{array}{c}
          e_1(x) \\
          \vdots \\
           e_s(x) \\
            \end{array}
            \right)=\mathbf{M^{-1}} \left(
          \begin{array}{c}
          \varepsilon_1(x) \\
          \vdots \\
           \varepsilon_s(x). \\
            \end{array}
            \right).
\end{eqnarray*}
\end{theorem}

By using the idempotents, a criterion on $0$-normal elements has been given in \cite{ZF}:
$\alpha \in \F^{\times}_{Q}$ is a normal element for extension $\F_{Q}/ \F_{q} \ (Q=q^n, \gcd(n,q)=1)$
if and only if $E_i(\alpha) \neq 0 \ (1 \leq i \leq s).$ Now we present such type of criterion on $1$-normal
elements as an application of Theorem \ref{thm-main}.

By Theorem \ref{thm-main}, $\alpha \in \F^{\times}_{Q}$ is a $1$-normal element for extension
$\F_{Q}/ \F_{q} \ (Q=q^n, \gcd(n,q)=1)$
if and only if the minimal $q$-polynomial $M_{\alpha}(x)$ of $\alpha$ is $\varphi(\frac{x^n -1}{p_i(x)})$
where $p_i(x)$ is a factor of $x^n -1$ in $\F_q[x]$ with degree $1$ so that $p_i(x)=x-c$
for some $c \in \F^{\times}_{q},$ which means that $c^{q-1}=1.$ Moreover, $c^n=1$ so that $c^d=1$
where $d=\gcd (q-1, n).$ Let $\gamma$ be a primitive element of $\F_q$
so that $\F^{\times}_{q}= \langle \gamma \rangle.$ Let $n=ed$ and $\beta=\gamma^e.$ Then the zeros
of $x^n -1$ in $\F_q$ are $\beta^{\lambda} \ (0 \leq \lambda \leq d-1)$ and the decomposition of
$x^n -1$ in $\F_q[x]$ is
\begin{equation}\label{eqn-dec}
x^n -1=p_1(x)p_2(x)\cdots p_d(x) p_{d+1}(x) \cdots  p_{s}(x)
\end{equation}
where $p_{\lambda}(x)=x-\beta^{\lambda -1}$ for $1 \leq \lambda \leq d,$ and $\deg p_{\lambda}(x)\geq 2 $
for $\lambda \geq d+1.$

For $ 1 \leq \lambda \leq d,$
\begin{eqnarray*}
l_{\lambda}(x)&=&\frac{x^n -1}{x-\beta^{\lambda -1}}=\sum^{n-1}_{i=0}\beta^{(\lambda-1)(n-1-i)}x^i \\
&=& \sum^{n-1}_{i=0}\beta^{(\lambda-1)(-1-i)}x^i \ (\mbox{since} \ \beta^n=1).
\end{eqnarray*}
And
$$
L_{\lambda}(x)=\varphi(l_{\lambda}(x))=\sum^{n-1}_{i=0}\beta^{(\lambda-1)(-1-i)}x^{q^{i}}.
$$

Therefore
\begin{eqnarray*}
L_{\lambda}(\alpha)&=&\sum^{e-1}_{l=0}\sum^{d-1}_{r=0}\beta^{(\lambda-1)(-1-r)}\alpha^{q^{dl+r}} \quad ( \mbox{let} \  i=dl+r )\\
&=& \sum^{d-1}_{r=0} \beta^{(\lambda-1)(-1-r)}(\tr^{n}_{d}(\alpha))^{q^{r}},
\end{eqnarray*}
where $\mathrm{Tr}^{n}_{d}$ is the trace mapping from $\F_Q =\F_{q^{n}}$ to $\F_{q^{d}}.$ Particularly,
$L_1(\alpha)=\sum^{d-1}\limits_{r=0}(\tr^{n}_{d}(\alpha))^{q^{r}}
=\tr^{d}_{1}(\tr ^{n}_{d}(\alpha))=\tr(\alpha)$
where $\tr=\tr^{n}_{1}$ is the trace mapping from $\F_Q$ to $\F_q.$

From these discussions we get the following result.

\begin{theorem}\label{thm-nor}
Suppose that $Q=q^n, \gcd(n,q)=1.$ Let $\F^{\times}_{q}=\langle \gamma \rangle, d=\gcd(n,q-1),n=ed$
and $\beta=\gamma^e.$ Then $x^n -1$ is decomposed in $\F_q[x]$ as formula (\ref{eqn-dec}). For any $\alpha \in \F^{\times}_{Q},$
the following statements are equivalent to each other.

(I) $\alpha$ is a $1$-normal element for extension $\F_Q / \F_q;$

(II) The minimum $q$-polynomial of $\alpha$ is
$L_{\lambda}(x)=\sum^{d-1}\limits_{r=0}\beta^{(\lambda-1)(-1-r)}(\mathrm{Tr}^{n}_{d}(\alpha))^{q^{r}}$
for some $\lambda, 1 \leq \lambda \leq d,$ where $\mathrm{Tr}^{n}_{d}(x)=\sum^{e-1}\limits_{l=0}x^{q^{dl}}.$

(III) There exists just one $\lambda$ for $ 1 \leq \lambda \leq d$ such that
 $\sum^{d-1}\limits_{r=0} \beta^{(\lambda-1)(-1-r)}(\mathrm{Tr}^{n}_{d}(\alpha))^{q^{r}}=0$
and $E_{\lambda}(\alpha) \neq 0$ for all $ d+1 \leq \lambda \leq s,$ where $E_i(x)$ is the $q$-polynomial
corresponding to the idempotent $e_i(x).$

(IV) There exists just one $\lambda$ for $1 \leq \lambda \leq d $
such that $\sum^{d-1}\limits_{r=0} \beta^{(\lambda-1)(-1-r)}(\mathrm{Tr}^{n}_{d}(\alpha))^{q^{r}}=0$
and $\{ \alpha, \alpha^q, \ldots \alpha^{q^{n-2}}\}$ is $\F_q$-linear independent.
\end{theorem}

\begin{proof}
$\alpha$ is a $1$-normal element of $\F_Q / \F_q$ if and only if the degree of the minimum
$q$-polynomial $M_{\alpha}(x)$ of $\alpha$ is $q^{n-1}.$ Namely,
$M_{\alpha}(x)=L_{\lambda}(x)=\varphi(l_{\lambda}(x))$ where $l_{\lambda}(x)=\frac{x^n -1}{p_{\lambda}(x)}$
 for some $\lambda, 1 \leq \lambda \leq d.$ Therefore (I) and (II) are equivalent. The other
 equivalent relations can be derived from Theorem \ref{thm-main}.
\end{proof}

\begin{corollary}
Suppose that $Q=q^n$ and $\gcd(n,q(q-1))=1.$ Then $x^n -1=p_1(x)p_2(x) \cdots p_s(x)$
where $p_1(x)=x-1$ and $\deg p_{\lambda} \geq 2$ for $2 \leq \lambda \leq s$. The following statements
are equivalent to each other for $\alpha \in \F^{\times}_{Q}.$

(I) $\alpha$ is a $1$-normal element for extension $\F_Q / \F_q;$

(II) The minimum $q$-polynomial of $\alpha$ is $\tr(x)=\sum^{n-1}\limits_{i=0}x^{q^i};$

(III) $\tr(\alpha)=0$ and $E_{\lambda}(\alpha) \neq 0$ for all $2 \leq \lambda \leq s;$

(IV) $\tr (\alpha)=0$ and $\{ \alpha, \alpha^q, \ldots, \alpha^{q^{n-2}}\}$ is
$\F_q$-linear independent.
\end{corollary}

\begin{proof}
By assumption $\gcd(n,q(q-1))=1,$ we know that $1$ is the only element $c$ in $\F^{\times}_{q}$
such that $c^n=1.$ Then the conclusion is derived from Theorem \ref{thm-nor} directly.
\end{proof}

\section{Examples}\label{sec-three}
In this section we present examples to determine the normality and the minimum
$q$-polynomial of an element  $\alpha \in \F^{\times}_{Q}$ by using Lemma \ref{lem-cri}
and Theorem \ref{thm-main}.

\begin{example}
Let $p$ and $n$ be prime numbers, $ n \neq p,$ and $q=p^m.$ Suppose that the order of
$q$ in $\gZ^{\times}_{n}$ is $\varphi(n)=n-1.$ Namely, $\gZ^{\times}_{n}=\langle q \rangle.$
Then $x^n-1=p_1(x)p_2(x)$ where
$$ p_1(x)=x-1, p_2(x)=x^{n-1}+x^{n-2}+\cdots+x+1 $$
are irreducible polynomials in $\F_q[x].$ We get
\begin{eqnarray*}
l_1(x)&=& \frac{x^n -1}{p_1(x)}, \  l_2(x)=\frac{x^n -1}{p_2(x)}=p_1(x) \\
L_1(x)&=& \varphi(l_1(x))=\sum^{n-1}_{i=0}x^{q^i}=\tr(x),\  L_2(x)=\varphi(l_2(x))=x^q -x.
\end{eqnarray*}
\end{example}

From Lemma \ref{lem-cri} we get the following result.

\begin{theorem}\label{thm-exam1}
Let $p$ and $n$ be prime numbers, $ n \neq p,$ and $q=p^m, Q=q^n.$ Suppose that
$\gZ^{\times}_{n}=\langle q \rangle.$ For each $\alpha \in \F^{\times}_{Q},$
let $M_{\alpha}(x)$ be the minimal $q$-polynomial of $\alpha.$

(I) If $\alpha \notin \F_q$ and $\tr(\alpha) \neq 0,$ then $M_{\alpha}(x)=x^{q^n}-x$
and $\alpha$ is a ($0$-th) normal element for $\F_Q / \F_q.$

(II) If $\alpha \notin \F_q$ and $\tr(\alpha) =0,$ then $M_{\alpha}(x)=\tr(x)$
and $\alpha$ is a $1$-normal element for $\F_Q / \F_q.$

(III) If $\alpha \in \F^{\times}_{q},$ then $\tr(\alpha)=n\alpha \neq 0$ so that
$M_{\alpha}(x)=x^{q}-x$ and $\alpha$ is an $(n-1)$-normal element for $\F_Q / \F_q.$
\end{theorem}

\begin{example}
Let $p$ be a prime number, $q=p^m,n$ be an odd prime, $n \neq p.$
Suppose that the order of $q$ in $\gZ^{\times}_{n}$ is $l=\frac{\varphi(n)}{2}=\frac{n-1}{2}.$
Then there exists an integer $g$ such that $\gZ^{\times}_{n}=\langle g \rangle$ and
$q=g^2 \in \gZ^{\times}_{n}.$ Then
$$D=\langle q \rangle =\{ q^{\lambda}: 0 \leq \lambda \leq l-1\}=\{g^{2\lambda}:0 \leq \lambda \leq l-1 \}$$
is the subgroup of multiplicative group $\gZ^{\times}_{n}$ and the other coset is
$D^{\prime}=gD=\{g^{2\lambda+1} :0 \leq \lambda \leq l-1\}.$ We have the decomposition
$$x^n -1=p_1(x)p_2(x)p_3(x)$$
in $\F_q[x]$ where
\begin{equation}\label{eqn-sep}
p_1(x)=x-1,\  p_2(x)=\prod_{a\in D}(x-\zeta^a),\ p_3(x)=\prod_{a\in D^{\prime}}(x-\zeta^a)
\end{equation}
where $\zeta$ is an $n$-th primitive root of $1$ in the algebraic closure of $\F_q.$
It is not easy to get the polynomials
$p_i(x) \in \F_q[x], l_i(x)=\frac{x^n -1}{p_i(x)} \in \F_q[x]$
and $L_i(x)=\varphi(l_i(x)) \in \mathcal{F}_q[x]$ explicitly for $i=2$ and $3.$
Now we use the idempotents.  With the notations given in Section \ref{sec-two}, we have
$$\varepsilon_1(x)=1,\ \varepsilon_2(x)=\sum_{r \in D}x^r,\  \varepsilon_3(x)=\sum_{r \in D^{\prime}}x^r$$
$$
\mM=\left(
   \begin{array}{ccc}
     \varepsilon_1(1) & \varepsilon_1(\zeta) &  \varepsilon_1(\zeta^g) \\
     \varepsilon_2(1) &  \varepsilon_2(\zeta) & \varepsilon_2(\zeta^g) \\
     \varepsilon_3(1) &  \varepsilon_3(\zeta) & \varepsilon_3(\zeta^g) \\
   \end{array}
 \right)
=\left(
   \begin{array}{ccc}
     1 & 1 & 1 \\
     l & C & B \\
     l & B & C \\
   \end{array}
 \right)
$$
where $B=\sum\limits_{r \in D^{\prime}}\zeta^r \in \F_q, C=\sum\limits_{r \in D}\zeta^r \in \F_q.$
By Theorem \ref{thm-ZF}, $\det(M)=n(B-C) \neq 0.$ Then we get
$$
\mM^{-1}=\frac{1}{n(B-C)}\left(
   \begin{array}{ccc}
    B-C & B-C & B-C \\
    l(B-C) & C-l & l-B \\
     l(B-C) & l-B & C-l \\
   \end{array}
 \right),
$$
and
$$\left(
          \begin{array}{c}
          e_1(x) \\
         e_2(x) \\
         e_3(x)\\
            \end{array}
            \right)= M^{-1} \left(
          \begin{array}{c}
          \varepsilon_1(x) \\
          \varepsilon_2(x) \\
           \varepsilon_3(x) \\
            \end{array}
            \right). $$
Namely, we get
\begin{eqnarray}\label{eqn-exa}
e_1(x)&=&\frac{1}{n}(\varepsilon_1(x)+\varepsilon_2(x)+\varepsilon_3(x))=\frac{1}{n}\sum^{n-1}_{i=0}x^i,  \nonumber \\
e_2(x)&=&\frac{1}{n(B-C)}[l(B-C)+(C-l)\varepsilon_2(x)+(l-B)\varepsilon_3(x)],  \\
e_3(x)&=&\frac{1}{n(B-C)}[l(B-C)+(l-B)\varepsilon_2(x)+(C-l)\varepsilon_3(x)].\nonumber
\end{eqnarray}

Now we compute $B$ and $C$ by using the Legendre symbol
$$
(\frac{r}{n}) =  \left \{
\begin{array}{ll}
1, & \mbox{if} \ r \in D, \\
-1 & \mbox{if} \ r \in D^{\prime}.
\end{array}
\right.
$$
We have $B+C=\sum^{n-1}\limits_{r=1} \zeta^r =-1, B-C=\sum^{n-1}\limits_{r=1} (\frac{r}{n}) \zeta^r \in \F_q,$
where $B-C$ is the quadratic Gauss sum over $\F_n,$ but valued in $\F_q,$ instead of the complex number field
$\C.$ We have calculated $B$ and $C$ in \cite{ZF} as following result.
\end{example}

Case (I): $2 \nmid q.$ Let $n^{\ast}=(\frac{-1}{n})n,$ then
$$B=\frac{1}{2}(-1+\mu\sqrt{n^{\ast}}), C=\frac{1}{2}(-1-\mu\sqrt{n^{\ast}}) \quad (\mu=1 \ \mbox{or} -1).$$
Then by (\ref{eqn-exa}) we get

\begin{eqnarray*}
n \mu \sqrt{n^{\ast}}e_2(x)&=&l\mu \sqrt{n^{\ast}}+
\frac{n}{2}(\varepsilon_3(x)-\varepsilon_2(x))
-\frac{\mu \sqrt{n^{\ast}}}{2}(\varepsilon_3(x)+\varepsilon_2(x)) \\
n \mu \sqrt{n^{\ast}}e_3(x)&=&l\mu \sqrt{n^{\ast}}-
\frac{n}{2}(\varepsilon_3(x)-\varepsilon_2(x))
-\frac{\mu \sqrt{n^{\ast}}}{2}(\varepsilon_3(x)+\varepsilon_2(x))
\end{eqnarray*}
and
\begin{eqnarray}\label{eqn-lar}
n E_1(x)&=& \tr(x) \nonumber \\
2n\sqrt{n^{\ast}} E_2(x)&=& n\sqrt{n^{\ast}} x -\mu n \sum^{n-1}_{r=1}(\frac{r}{n})x^{q^r}-\sqrt{n^{\ast}} \tr(x) \\
2n\sqrt{n^{\ast}} E_3(x)&=& n\sqrt{n^{\ast}} x +\mu n \sum^{n-1}_{r=1}(\frac{r}{n})x^{q^r}-\sqrt{n^{\ast}} \tr(x).\nonumber
\end{eqnarray}

Case (II): $2 \mid q.$ Then $B+C=B-C=1$ and
$$
\{ B,C \} =  \left \{
\begin{array}{ll}
\{ 0,1\}, & \mbox{if} \ n \equiv \pm1 (\bmod 8), \\
\{\omega,\omega +1\} & \mbox{if} \ n \equiv \pm 3(\bmod 8),
\end{array}
\right.
$$
where $\omega \in \F_4 \backslash \{0,1\}.$ Then by (\ref{eqn-exa}) we get
\begin{eqnarray*}
ne_2(x)&=& l+(l+B)(\varepsilon_2(x)+\varepsilon_3(x))+\varepsilon_2(x)
=l\sum^{n-1}_{r=0}x^r +B\sum^{n-1}_{r=1}x^r+\sum_{r \in D}x^r \\
ne_3(x)&=& l\sum^{n-1}_{r=0}x^r +C\sum^{n-1}_{r=1}x^r+\sum_{r \in D}x^r,
\end{eqnarray*}
and
\begin{eqnarray}\label{eqn-res}
n E_1(x)&=& \tr(x) \nonumber \\
n E_2(x)&=& l \tr(x)+B(\tr(x)+x)+\sum_{r \in D}x^{q^r}\\
n E_3(x)&=& l \tr(x)+C(\tr(x)+x)+\sum_{r \in D}x^{q^r} \quad (C=B+1).\nonumber
\end{eqnarray}

Now we determine the normality of any element $\alpha \in \F^{\times}_Q.$

\begin{theorem}\label{thm-exam2}
Let $p$ and $n$ be distinct prime numbers, $ n \geq 3, q=p^m, Q=q^n.$
Suppose that $\gZ^{\times}_{n}=\langle g \rangle$ and $q=g^2 \in \gZ^{\times}_{n}$
so that the order of $q$ in the multiplicative group $\gZ^{\times}_{n}$ is $l=\frac{n-1}{2}.$
Let $$D=\langle q \rangle =\{ g^{2 \lambda}: 0 \leq \lambda \leq l-1\}.$$
Then $x^n -1=p_1(x)p_2(x)p_3(x)$ where $p_i(x) (1 \leq i \leq 3)$ are the monic
irreducible factors of $x^n -1$ in $\F_q[x]$ defined by (\ref{eqn-sep}).
 Let $P_i(x)=\varphi(p_i(x)) \ (1 \leq i \leq 3).$ For $\alpha \in \F^{\times}_Q,$
let $M_{\alpha}(x)$ be the minimum $q$-polynomial of $\alpha$,
and $\alpha$ is a $k$-normal element for $\F_Q / \F_q.$

Case(I): $2 \nmid q.$ Let $\delta =\sum^{n-1}\limits_{r=0}(\frac{r}{n})\alpha^{q^r}$

$(a)$ If $\alpha \in \F^{\times}_q,$ then $M_{\alpha}(x)=x^q -x$ and $k=n-1.$

$(b)$ Suppose that $\alpha \in \F_Q \backslash \F_q$ and $\tr(\alpha)=0.$

If $\sqrt{n^{\ast}}\delta  \notin \{ \pm n \alpha\},$ then $M_{\alpha}(x)=\tr(x)=\sum^{n-1}\limits_{i=0}x^{q^i}$
and $k=1.$ Otherwise $M_{\alpha}(x)=P_2(x) $ or $P_3(x) $ and $k=l+1=\frac{n+1}{2}.$

$(c)$ Suppose that $\alpha \in \F_Q \backslash \F_q$ and $\tr(\alpha) \neq 0.$

If $\sqrt{n^{\ast}}\delta  \notin \{ \pm (n \alpha-\tr(\alpha)  \},$ then $M_{\alpha}(x)=x^{q^n}-x$
and $k=0$ ($\alpha$ is a normal element for $\F_Q / \F_q.$) Otherwise $M_{\alpha}(x)=P_i(x^q-x)=P_i(x)^q -P_i(x) $
for $i=2$ or $ 3$ and $k=l=\frac{n-1}{2}.$

Case (II): $2 \mid q.$ Let $\varepsilon=\sum\limits_{r \in D}\alpha^{q^r}, \omega \in \F_4 \backslash \{0,1\}$ and
$$
 B =  \left \{
\begin{array}{ll}
0, & \mbox{if} \ n \equiv \pm1 (\bmod 8), \\
\omega, & \mbox{if} \ n \equiv \pm 3(\bmod 8).
\end{array}
\right.
$$

$(a)$ If $\alpha \in \F^{\times}_q,$ then $M_{\alpha}(x)=x^q -x$ and $k=n-1.$

$(b)$ Suppose that $\alpha \in \F_Q \backslash \F_q$ and $\tr(\alpha)=0.$
If $\varepsilon \notin \{ B\alpha, (B+1)\alpha\},$
then $M_{\alpha}(x)=\tr(x)=\sum^{n-1}\limits_{i=0}x^{q^i}$ and $k=1.$
Otherwise $M_{\alpha}(x)=P_2(x) $ or $P_3(x) $ and $k=l+1=\frac{n+1}{2}.$

$(c)$ Suppose that $\alpha \in \F_Q \backslash \F_q$ and $\tr(\alpha) \neq 0.$
If $\varepsilon \notin \{ l \tr(\alpha)+B(\tr(\alpha)+\alpha), l \tr(\alpha)+(B+1)(\tr(\alpha)+\alpha)\},$
then $M_{\alpha}(x)=x^{q^n}-x$ and $k=0$ ($\alpha$ is a normal element for $\F_Q / \F_q.$)
Otherwise $M_{\alpha}(x)=P_i(x)^q +P_i(x) $
for $i=2$ or $ 3$ and $k=l=\frac{n-1}{2}.$
\end{theorem}

\begin{proof}
(I) For $2 \nmid q, \alpha \in \F^{\times}_Q,$ formula (\ref{eqn-lar}) gives that
\begin{eqnarray*}
E_1(\alpha)=0 & \Leftrightarrow & \tr(\alpha)=0 \\
E_2(\alpha)=0 & \Leftrightarrow & n\alpha-\tr(\alpha)-\mu \sqrt{n^{\ast}}\delta=0 \\
E_3(\alpha)=0 & \Leftrightarrow & n\alpha-\tr(\alpha)+\mu \sqrt{n^{\ast}}\delta=0 \ (\mu=1 \mbox{or} -1).
\end{eqnarray*}
If $\alpha \in \F^{\times}_q,$
 then $\tr(\alpha)=n\alpha \neq 0,
 \delta=\sum^{n-1}\limits_{r=0}(\frac{r}{n})\alpha^{q^r}=\alpha \sum^{n-1}\limits_{r=0}(\frac{r}{n})=0.$
Therefore $E_1(\alpha) \neq 0$ and $E_2(\alpha)=E_3(\alpha)=0.$ By Theorem \ref{thm-main},
$M_{\alpha}(x)=P_1(x)=x^q-x$ and $k=n-1.$ If $\alpha \in \F_Q \backslash \F_q,$ then $n \alpha -\tr(\alpha) \neq 0
 ( \mbox{otherwise} \ \alpha=\frac{1}{n}\tr(\alpha) \in \F_q).$ If $E_2(\alpha)=E_3(\alpha)=0,$
then $n\alpha-\tr(\alpha)=\sqrt{n^{\ast}}\delta=-\sqrt{n^{\ast}}\delta$ which implies that $\delta =0$
and $n \alpha =\tr(\alpha),$ contradiction.

Therefore at most one of $E_2(\alpha)$ and $E_3(\alpha)$
is zero. And $E_i(\alpha)=0$ for $i=2$ or $3$ if and only if
$\sqrt{n^{\ast}}\delta \in \{ \pm (n\alpha-\tr(\alpha))\}.$ When $\tr(\alpha)=0,$ then $E_1(\alpha)=0.$
If $\sqrt{n^{\ast}}\delta \notin \{ \pm n\alpha \},$ then $E_2(\alpha) \neq 0 \neq E_3(\alpha)$
and $M_{\alpha}(x)=\varphi(p_2(x)p_3(x))=\varphi(\sum^{n-1}\limits_{i=0}x^i)=\sum^{n-1}\limits_{i=0}x^{q^i},k=n-(n-1)=1.$
If $\sqrt{n^{\ast}}\delta =n \alpha$ or $-n\alpha$
(namely, $\delta=\sqrt{n^{\ast}} \alpha$ or $-\sqrt{n^{\ast}} \alpha$).
Then $M_{\alpha}(x)=P_i(x) $
where $i=2$ or $ 3$ and $k=n-l=l+1.$ When $\tr(\alpha) \neq 0, E_1(\alpha) \neq 0.$ If
$\sqrt{n^{\ast}}\delta \notin \{ \pm (n\alpha-\tr(\alpha))\},$ then $E_2(\alpha) \neq 0 \neq E_3(\alpha)$
and $M_{\alpha}(x)=\varphi(x^n -1)=x^{q^n}-x, k=0.$
Otherwise $M_{\alpha}(x)=\varphi(p_i(x)p_1(x))=P_i(x) \otimes(x^q -x)=P_i(x^q -x)=P_i(x)^q -P_i(x)$
for $i=2$ or $3.$ This completes the proof of Case (I). Similarly we can prove Theorem \ref{thm-exam2} for Case (II)
by using Theorem \ref{thm-main} and formula (\ref{eqn-res}).
\end{proof}

\begin{example}(General case)
Let $p$ and $n$ be distinct prime numbers, $n \geq 3, q=p^m, Q=q^n.$
Let $f$ be the order of $q$ in the multiplicative group $\gZ^{\times}_n,$ then $n-1=ef$
and there exists $g \in \gZ^{\times}_n$ such that $\gZ^{\times}_n=\langle g \rangle$
and $q=g^e \in \gZ^{\times}_n. \  C=\langle q \rangle $ is a subgroup of $\gZ^{\times}_n$
and the coset of $C$ in $\gZ^{\times}_n$ are
$$
C_{\lambda}=g^{\lambda}C=\{ g^{\lambda+ie}: 0 \leq i \leq f-1\} \quad (0 \leq \lambda \leq e-1).
$$

Let $\zeta$ be an $n$-th primitive root of $1, \F_q(\zeta)=\F_{q^f}.$ Then
\begin{equation}
x^n -1 = p_{\ast}(x)p_0(x) \cdots p_{e-1}(x)
\end{equation}
where $p_{\ast}(x)=x-1$ and for $0 \leq \lambda \leq e-1, p_{\lambda}=\sum\limits_{a \in C_{\lambda}}(x-\zeta^a)$
is an irreducible polynomial in $\F_q[x].$ Therefore

\begin{equation}
\varepsilon_{\ast}=1, \varepsilon_{\lambda}(x)=\sum_{a \in C_{\lambda}}x^a (\bmod x^n -1) \ (0 \leq \lambda \leq e-1).
\end{equation}

\end{example}

Let $\varepsilon_{\lambda}=\varepsilon_{\lambda}(\zeta)=\sum\limits_{a \in C_{\lambda}} \zeta^a (0 \leq \lambda \leq e-1).$
We know that $\varepsilon_{\lambda} \in \F_q$ is the Gauss periods of order $e$ and for $\alpha_j=\zeta^{g^j},$
$$
\varepsilon_{\lambda}(\alpha_j)=\sum_{a \in C_{\lambda}}\zeta^{ag^j}=\varepsilon_{\lambda+j} \ (\lambda, j \in \gZ_e).
$$

Therefore
$$
\mM=\left(
  \begin{array}{ccccc}
   1 & 1 & 1 & \cdots & 1 \\
    f & \varepsilon_0 & \varepsilon_1 & \cdots & \varepsilon_{e-1} \\
    f & \varepsilon_1 & \varepsilon_2 & \cdots & \varepsilon_0 \\
     \vdots & \vdots & \vdots & & \vdots \\
       f &\varepsilon_{e-1} &\varepsilon_0 & \cdots & \varepsilon_{e-2} \\
        \end{array}
        \right).
$$

By using the equalities
\begin{equation}
\sum^{e-1}_{\lambda=0}\varepsilon_{\lambda}=\sum^{e-1}_{\lambda=0}\sum_{a \in C_{\lambda}}\zeta^a=\sum^{n-1}_{a=0}\zeta^a=-1
\end{equation}
and
\begin{eqnarray*}
\sum^{e-1}_{\lambda=0}\varepsilon_{\lambda} \varepsilon_{\lambda+j}&=&\sum^{e-1}_{\lambda=0}\sum_{a,b \in \mathrm{C}}\zeta^{ag^{\lambda}+bg^{\lambda+j}}\\
&=&\sum^{e-1}_{\lambda=0}\sum_{a \in \mathrm{C}}\sum_{d \in \mathrm{C}}\zeta^{ag^{\lambda}(1+dg^j)} \ (\mbox{let} \  d=ba^{-1}) \\
&=&\left \{
\begin{array}{ll}
n-f , \  \mbox{if} \ -1 \in \mathrm{C_j}  \ (\mbox{namely, if} \  j=0 \ \mbox{for even} \  f\  \mbox{and} \  j=\frac{e}{2} \ \mbox{for odd} \ f)\\
-f, \ \ \mbox{otherwise}.
\end{array}
\right.
\end{eqnarray*}
we get
$$
\mM^{-1}= \frac{1}{n}\left(
\begin{array}{ccccc}
 1 & 1 & 1 & \cdots & 1 \\
  f & \varepsilon_c & \varepsilon_{c+1} & \cdots & \varepsilon_{c-1} \\
  f & \varepsilon_{c+1} & \varepsilon_{c+2} & \cdots & \varepsilon_c \\
  \vdots & \vdots & \vdots & & \vdots \\
  f &\varepsilon_{c-1} &\varepsilon_c & \cdots & \varepsilon_{c-2} \\
 \end{array}
 \right).
$$
where $c \equiv \frac{ef}{2}(\bmod e).$ Namely, $c=0$ for even $f$ and $c=\frac{f}{2}$
for odd $f.$ By Theorem \ref{thm-ZF}, we have for $\alpha \in \F^{\times}_Q \ (Q=q^n),$
$$\left(
          \begin{array}{c}
          e_{\ast}(x) \\
          e_0(x)\\
          \vdots \\
           e_{e-1}(x) \\
            \end{array}
            \right)= \mM^{-1} \left(
          \begin{array}{c}
          1 \\
          \varepsilon_0(x)\\
          \vdots \\
           \varepsilon_{e-1}(x) \\
            \end{array}
            \right)=\frac{1}{n}\left(
          \begin{array}{c}
         \sum^{n-1}\limits_{i=0}x^i \\
          f+\sum^{e-1}\limits_{\lambda=0}\varepsilon_{\lambda}(x)\varepsilon_{\lambda+c} \\
          \vdots \\
           f+\sum^{e-1}\limits_{\lambda=0}\varepsilon_{\lambda}(x)\varepsilon_{\lambda+c-1} \\
            \end{array}
            \right) $$
Namely, $E_{\ast}(\alpha)=\sum^{n-1}\limits_{i=0}\alpha^{q^i}=\tr(\alpha)$ and for $0 \leq i \leq e-1,$
\begin{equation}
E_i(\alpha)=\frac{1}{n}(f\alpha+\sum^{e-1}_{\lambda=0}\varepsilon_{\lambda+c+i}\sum_{a \in C_{\lambda}}\alpha^{q^a}).
\end{equation}
From these computation and Theorem \ref{thm-main}, we get the following result.

\begin{theorem}\label{thm-exam3}
Let $p$ and $n$ be distinct prime numbers, $n \geq 3, q=p^m, Q=q^n.$
Let $f$ be the order of $q$ in the multiplicative group $\gZ^{\times}_n, n-1=ef$
Then there exists $g \in \gZ^{\times}_n$ such that $\gZ^{\times}_n=\langle g \rangle$
and $q=g^e \in \gZ^{\times}_n.$ Let
$$
C=\langle q \rangle =\{ q^{i}: 0 \leq i \leq f-1\} =\{ g^{ie}: 0 \leq i \leq f-1\}$$
and
$C_{\lambda}=g^{\lambda}C \ (0 \leq \lambda \leq e-1)$ are the cyclotomic cosets of $C$ in $\gZ^{\times}_n.$
Let $\zeta$ be an $n$-th primitive root of $1, \F_q(\zeta)=\F_{q^f}.$ Then
$x^n-1$ is decomposed in $\F_q[x]$ by
$$
x^n -1 = p_{\ast}(x)p_0(x) \cdots p_{e-1}(x)
$$
where $p_{\ast}(x)=x-1$ and for $0 \leq \lambda \leq e-1, p_{\lambda}=\sum\limits_{a \in C_{\lambda}}(x-\zeta^a)$
is an irreducible polynomial in $\F_q[x],$ and $P_{\lambda}=\varphi(p_{\lambda}).$ Let
\begin{eqnarray*}
\varepsilon_{\lambda}&=&\sum_{a \in C_\lambda}\zeta^{a}\\
c &=&\left \{
\begin{array}{ll}
0 , \  \mbox{if} \  f \  \mbox{is even}\\
\frac{e}{2}, \ \mbox{if} \ f \ \mbox{is odd}.
\end{array}
\right.
\end{eqnarray*}
For $\alpha \in \F^{\times}_Q,$ let
\begin{equation}
\mathcal{S}=\{i:0 \leq i \leq e-1, \ \sum^{e-1}_{\lambda=0}\varepsilon_{\lambda+i+c}\sum_{a \in C_{\lambda}}\alpha^{q^a} \neq -fa\}
\end{equation}
$M_{\alpha}(x)$ be the minimum $q$-polynomial of $\alpha$ and $\alpha$ be a $k$-normal element for $\F_Q / \F_q.$

(I) When $\alpha \in \F^{\times}_q,$ then $M_{\alpha}(x)=x^q -x$ and $k=n-1.$

(II) When $\alpha \in \F_Q / \F_q$ and $\tr(\alpha)=0,$ then $M_{\alpha}(x)=\bigotimes\limits_{\lambda \in \mathcal{S}}P_{\lambda}(x)$
and $k=n-f |\mathcal{S}|=f(e-|\mathcal{S}|)+1.$

(III) When $\alpha \in \F_Q / \F_q$ and $\tr(\alpha) \neq 0,$
 then $M_{\alpha}(x)=\bigotimes\limits_{\lambda \in \mathcal{S}}P_{\lambda}(x^q-x)$
and $k=n-1-f |\mathcal{S}|=f(e-|\mathcal{S}|).$
\end{theorem}

\begin{proof}
When $\alpha \in \F^{\times}_q, E_{\ast}=\tr(\alpha)=n\alpha \neq 0$ and for $ 0 \leq i \leq e-1,$
by (16)
$$
E_i(\alpha)=\frac{1}{n}(f \alpha +\sum^{e-1}_{\lambda=0}\varepsilon_{\lambda+c+i}\sum_{a \in C_{\lambda}})
=\frac{1}{n}(f \alpha +f \alpha \sum^{e-1}_{\lambda=0}\varepsilon_{\lambda}=0.
$$
Therefore $M_{\alpha}(x)=P_1(x)=x^q-x$ and $k=n-1.$

When $\alpha \in \F_Q / \F_q, E_{\ast}(\alpha)=\tr(\alpha) $ and for $ 1 \leq i \leq e-1,$
by (16) and (17),
$$E_i(\alpha) \neq 0  \Longleftrightarrow i \in \mathcal{S}.$$
Therefore $M_{\alpha}(x)=\bigotimes\limits_{\lambda \in \mathcal{S}}P_{\lambda}(x)$ and
$k=n-f |\mathcal{S}|$ if $ \tr(\alpha)=0$ and
$M_{\alpha}(x)=(\bigotimes\limits_{\lambda \in \mathcal{S}}P_{\lambda}(x))$
$\bigotimes P_1(x)
=\bigotimes\limits_{\lambda \in \mathcal{S}}P_{\lambda}(x^q -x), k=n-f |\mathcal{S}|-1=f(e-|\mathcal{S}|)$
if $\tr(\alpha) \neq 0.$
\end{proof}

\begin{remark}
(1) $\varepsilon_i=\sum\limits_{a \in C_1}\zeta^a  \ ( 0 \leq i \leq e-1)$ are Gauss periods
of order $e$ over $\F_n,$ but valued in $\F_q.$ They can be computed as usual Gauss periods
valued in complex number field $\C$ for small $e$ and semiprimitive case.
For $e=1$ and $2,$ we get Theorem \ref{thm-exam1} and \ref{thm-exam2} respectively.

(2) For $q=2,(e,n)=(3,7),(5,31),(7,127)$ or $q=4,e=3,n=7. \varepsilon_i  \in \F_2$
and by (16),

$$
\sum^{e-1}_{i=0}\varepsilon_{i}\varepsilon_{i+j} =  \left \{
\begin{array}{ll}
1, & \mbox{if} \ j=0, \\
0 & \mbox{if} \ 1 \leq j \leq e-1,
\end{array}
\right.
$$
which means that the circulant matrix over $\F_2,$
$$
\mM=\left(
  \begin{array}{cccc}
   \varepsilon_0  & \varepsilon_1 &\cdots & \varepsilon_{e-1} \\
    \varepsilon_1 & \varepsilon_2 &  \cdots & \varepsilon_{0} \\
   \vdots & \vdots &  & \vdots \\
      \varepsilon_{e-1} &\varepsilon_0 & \cdots & \varepsilon_{e-2} \\
        \end{array}
        \right).
$$
is orthogonal. Jungnickel et al. \cite{JBG} obtained a formula on the
number of orthogonal circulant $ e \times 2$ matrix over $\F_q.$ From this
formula we know that for $q=2, (e,n)=(3,7),(5,31),(7,127)$ or $q=4,(e,n)=(3,7)$
$ (\varepsilon_0,\varepsilon_1,\ldots ,\varepsilon_{e-1})$ is a cyclic shift of
$(1,0,\ldots,0).$ Let $\varepsilon_t =1,$
then $\sum^{e-1}\limits_{\lambda=0}\varepsilon_{\lambda+i+c}\sum\limits_{a \in C_{\lambda}}\alpha^{q^a}
\sum\limits_{a \in C_{\lambda^{\prime}}}\alpha^{q^a}$ where $\lambda^{\prime}=t+i+c.$ Therefore let
$$\mathcal{S^{\prime}}=\{ \lambda^{\prime}:
 0 \leq \lambda^{\prime} \leq e-1, \sum_{a \in C_{\lambda^{\prime}}}\alpha^{q^a}\} \neq -fa.$$
Then $\mathcal{S^{\prime}}=\mathcal{S}+t=\{s+t, s \in \mathcal{S}\},|\mathcal{S}|=|\mathcal{S^{\prime}}| ,$
and Theorem \ref{thm-exam3} can be stated in term of $\mathcal{S^{\prime}}$ in stead of $\mathcal{S}.$
\end{remark}

\section*{Acknowledgements}

 K.Feng's research was supported by the Natural Science Foundation of China under Grant No:11571107 and 11471178.

\end{document}